\documentclass{amsproc}

\usepackage{amssymb}
\usepackage{mathrsfs}
\usepackage{enumerate}

\newcommand{\A}{\mathop{\mathscr A}\nolimits}
\newcommand{\Om}{\Omega}
\newcommand{\om}{\omega}
\newcommand{\RR}{\mathbb{R}}
\newcommand{\CC}{\mathbb{C}}
\newcommand{\Real}{\mathop{\mathscr Re}\nolimits}
\newcommand{\Imm}{\mathop{\mathscr Im}\nolimits}
\newcommand{\elle}{\mathop{\mathscr L}\nolimits}
\newcommand{\lan}{\langle}
\newcommand{\ran}{\rangle}
\newcommand{\Cspt}{C_{0}}

\newcommand{\dive}{\mathop{\rm div}}
\newcommand{\de}{\partial}

\newcommand\B{\mathop{\mathscr B}\nolimits}

\newcommand\Cm{\mathop{\mathscr C}\nolimits}
\newcommand\Dm{\mathop{\mathscr D}\nolimits}

\newtheorem{theorem}{Theorem}[section]
\newtheorem{lemma}[theorem]{Lemma}
\newtheorem{corollary}[theorem]{Corollary}

\theoremstyle{definition}

\newtheorem{example}[theorem]{Example}

\theoremstyle{remark}
\newtheorem{remark}[theorem]{Remark}

\numberwithin{equation}{section}

\begin{document}

\title{The $L^p$-dissipativity of certain  differential and
integral operators}


\author[A. Cialdea]{Alberto Cialdea}
\address{Dipartimento di Matematica, Economia ed Informatica, University of Basilicata, V.le dell'Ateneo Lucano 10, 85100, Potenza, Italy}
\curraddr{}
\email{cialdea@email.it}
\thanks{}

\author[V. Maz'ya]{Vladimir Maz'ya}
\address{Department of Mathematical Sciences, M\& O
  Building, University of Liverpool, Liverpool L69 7ZL, UK.
Department of Mathematics, Link\"oping University,
SE-581 83, Link\"oping, Sweden.
RUDN University,
6 Miklukho-Maklay St, Moscow, 117198, Russia.}
\curraddr{}
\email{vladimir.mazya@liu.se}
\thanks{The publication was supported by the Ministry of Education and Science of the Russian Federation (Agreement No. 02.a03.21.0008).}

\dedicatory{In Memory of Selim G. Krein}


\date{}

\begin{abstract}
The first part of the paper is a survey of some of the results previously obtained by the
authors concerning the $L^p$-dissipativity of scalar and matrix partial
differential operators. In the second part we 
give new necessary and, separately, sufficient conditions
for the $L^p$-dissipativity of the ``complex oblique derivative'' operator.
In the case of real coefficients we provide a necessary and sufficient condition. 
We prove also the $L^p$-positivity for a certain class of
integral operators.
\end{abstract}

\maketitle


\section{Introduction}

In a series of papers \cite{cialmaz1,cialmaz2,cialmaz3,cialmaz17} we have studied the
problem of characterizing the $L^p$-dissipativity
of scalar and matrix partial differential operators.
The main result we have obtained is that the algebraic condition
\begin{equation}\label{eq:intro0}
|p-2|\, |\lan \Imm\A\xi,\xi\ran| \leq 2 \sqrt{p-1}\,
	\lan \Real\A\xi,\xi\ran
\end{equation}
for any $\xi\in\RR^n$,  is necessary and sufficient for the $L^p$-dissipativity
of the Dirichlet problem for the scalar differential operator
$\nabla^t(\A\nabla)$, where $\A$ is a matrix whose entries
are complex measures, not necessarily absolutely continuous, and whose imaginary parts is symmetric. 
Specifically we have proved that condition \eqref{eq:intro0}
is necessary and sufficient for the $L^p$-dissipativity
of the related sesquilinear form
$$
\elle(u,v)= \int_\Omega \lan \A\nabla u, \nabla v\ran\, .
$$

Such condition characterizes the $L^p$-dissipativity
individually, for each $p$, while in the literature previous
results dealt with the $L^p$-dissipativity for any $p\in [1,+\infty)$.
Later on we have  considered more general operators. Our results
are described in the monograph \cite{cialmazbook}.

We remark that, if $\Imm\A$ is symmetric, \eqref{eq:intro0} is equivalent to the
the condition
\begin{equation}\label{intro:form}
 \frac{4}{p\,p'}\lan \Real \A\xi,\xi\ran + \lan \Real \A\eta,\eta\ran
       +2 \lan (p^{-1}\Imm\A + p'^{-1}\Imm\A^{*}) \xi,\eta \ran \geq 0
\end{equation}
for any $\xi, \eta\in\RR^n$. If the matrix $\Imm\A$ is not symmetric,
condition \eqref{intro:form} is only sufficient for the $L^p$-dissipativity
of the corresponding form. 

Let us consider the
class of partial differential operators of the second order whose principal part is such that the form \eqref{intro:form} is not merely
non-negative, but strictly positive. This class of operators,
which could be called \textit{$p$-strongly elliptic}, was recently considered by  Carbonaro  and Dragi\v{c}evi\'c \cite{carbonaro}, and Dindo\v{s} and Pipher \cite{dindos}.

In what follows, saying the $L^p$-dissipativity of an operator $A$, we mean the $L^p$-dissipativity of the corresponding form $\elle$, just to 
simplify the terminology.

In the present paper, after surveying some of our
more recent results for systems which are not contained in \cite{cialmazbook}, we give some new  
theorems.

The first ones concern the ``complex oblique derivative''
operator
$$
\lambda\cdot\nabla u=
\frac{\partial u}{\partial x_{n}} +
\sum_{j=1}^{n-1}a_{j}\frac{\partial u}{\partial x_{j}}
$$
where $\lambda=(1,a_1,\ldots,a_{n-1})$ and $a_j$ are  complex valued 
functions. We give necessary and, separately, sufficient 
conditions under which such boundary operator is $L^p$-dissipative
on $\RR^{n-1}$.
If the coefficients $a_j$ are real valued, 
we provide a necessary and sufficient condition.

The last result concerns a class of integral
operators which can be
written as
\begin{equation}\label{intro:1}
\int_{\RR^{n}}^*[u(x)-u(y)]\, K(dx,dy)
\end{equation}
where the integral has to be understood as a principal value in the sense of Cauchy
and the kernel $K(dx,dy)$ is a Borel positive measure defined 
on $\RR^n\times\RR^n$ satisfying certain conditions. The class
of operators we consider includes the fractional powers of Laplacian
$(-\Delta)^s$, with $0<s<1$. We establish the $L^p$-positivity
of operator \eqref{intro:1}, extending in this way a result we
obtained in \cite[p.230--231]{cialmazbook}.

The paper is organized as follows. 
Section \ref{sec:scalar} presents a review of our main results concerning
scalar differential operators of the second order.

Section \ref{sec:systems} is dedicated to systems. In particular we describe some necessary and sufficient conditions for
the $L^p$-dissipativity of systems of partial differential operators
of the first order, recently obtained
 in \cite{cialmaz17}.

The topic of Section \ref{sec:elast} is elasticity system. After recalling the necessary and sufficient conditions we previously obtained in the
planar case, we describe sufficient conditions holding in 
any dimension and proved in \cite{cialmaz3}.  Finding necessary and sufficient conditions for the
$L^p$-dissipativity of elasticity system in the
three-dimensional case is still an open problem.

In Sections \ref{sec:oblique} and \ref{sec:intop} 
we prove the above mentioned  results concerning the ``complex oblique
derivative'' operator and the operator \eqref{intro:1} respectively.

\section{The scalar operators}\label{sec:scalar}

Let $\Om$ be 
an open set in $\RR^{n}$. Consider the
sesquilinear form
$$
    {\elle}(u,v)=\int_{\Om}(\lan \A\nabla u,\nabla v\ran -\lan {\bf b}\nabla u,v\ran
    +\lan u,\overline{\bf c}\nabla v\ran
    -a\lan u,v\ran)\,
$$
defined on $\Cspt^{1}(\Om)\times \Cspt^{1}(\Om)$.
Here  $\A$ is a $n\times n$ matrix
function with complex valued entries $a^{hk}\in (\Cspt(\Om))^{*}$,
${\bf b}=(b_{1},\ldots,b_{n})$ and ${\bf c}=(c_{1},\ldots,c_{n})$
are complex
valued vectors
with $b_{j}, c_{j}\in (\Cspt(\Om))^{*}$ and
$a$ is a complex valued scalar distribution in $(\Cspt^{1}(\Om))^{*}$.
The symbol $\lan \cdot, \cdot \ran$  denotes the inner product either in
$\CC^{n}$ or in $\CC$.

In what follows, if $p\in(1,\infty)$, $p'$ denotes its conjugate exponent $p/(p-1)$.

The integrals appearing in this definition have to be understood in
a proper way. The entries $a^{hk}$ being measures, the meaning of the
first term is
$$
\int_{\Om}\lan \A\nabla u,\nabla v\ran = 
\int_{\Om}\de_k u\, \de_h \overline{v}\, da^{hk}\, .
$$

Similar meanings have the terms involving ${\bf b}$ and ${\bf c}$. 
Finally, the last term is the action of the distribution $a\in 
(\Cspt^{1}(\Om))^{*}$ on
the functions $u\, \overline{v}$ belonging to $\Cspt^{1}(\Om)$.

We say that the form ${\mathscr L}$ is
\textit{$L^{p}$-dissipative} ($1<p<\infty$) if
for all $u\in \Cspt^{1}(\Om)$
\begin{eqnarray}
    \Real {\mathscr L}(u, |u|^{p-2}u) \geq 0
    \qquad \hbox{\rm if}\ p\geq 2;
    \label{eq:defdis1}\\
    \Real {\mathscr L}(|u|^{p'-2}u, u) \geq 0 \qquad
    \hbox{\rm if}\ 1<p< 2
      \label{eq:defdis2}
\end{eqnarray}
(we use here that $|u|^{q-2}u\in \Cspt^{1}(\Om)$ for $q\geq 2$ and
$u\in \Cspt^{1}(\Om)$).

The form ${\mathscr L}$ is related to the operator
\begin{equation}
    Au =\dive(\A\nabla u) + {\bf b}\nabla u + \dive({\bf c}u)+au.
    \label{eq:N}
\end{equation}
where $\dive$ denotes the divergence operator.
The operator $A$ acts from $\Cspt^{1}(\Om)$ to $(\Cspt^{1}(\Om))^{*}$ through the
relation
$$
\elle(u,v)=-\int_{\Om}\lan Au,v\ran
$$
for any $u,v \in \Cspt^{1}(\Om)$. The integration is understood in the sense of distributions.

As we already remarked, saying the $L^p$-dissipativity of the operator $A$, we mean the$L^p$-dissipativity of the form $\elle$.

The following Lemma provides a necessary and sufficient
condition for the $L^p$-dissipativity of the form $\elle$.

\begin{lemma}[\cite{cialmaz1}]\label{lemma:1}
The operator $A$ is $L^{p}$-dissipative
    if and only if for all $v\in \Cspt^{1}(\Om)$
    \begin{gather*}
  \Real \int_{\Om}\Big[ \lan \A\nabla v,\nabla v\ran -
	(1-2/p)\lan (\A-\A^{*})\nabla(|v|),|v|^{-1}\overline{v}\nabla v\ran  -\\
	(1-2/p)^{2}\lan \A \nabla(|v|),\nabla(|v|)\ran
\Big]
+
\int_{\Om}\lan\Imm ({\bf b}+{\bf c}),\Imm(\overline{v}\nabla
v)\ran\,   +
\\
\int_{\Om}\Real(\dive ({\bf b}/p - {\bf c}/p') - a
)|v|^{2}
	\geq 0.
\end{gather*}

  Here and in what follows the integrand is extended by zero
    on the set where $v$ vanishes.
\end{lemma}

This result has several consequences.
The first one is a necessary condition for
the $L^p$-dissipativity.

\begin{corollary}[\cite{cialmaz1}]
    If the operator $A$ is $L^{p}$-dissipative, we have
   \begin{equation}
        \lan \Real \A \xi,\xi\ran \geq 0
       \label{eq:mainpart}
   \end{equation}
(in the sense of measures) for any $\xi\in\RR^{n}$.
\end{corollary}

Obviously condition \eqref{eq:mainpart} is not sufficient for
the $L^p$-dissipativity of the form $\elle$.

\begin{corollary}[\cite{cialmaz1}]\label{cor:1}
    Let $\alpha,\beta$ two real constants. If
   \begin{equation}
   \begin{gathered}\label{polyn}
  \frac{4}{p\,p'}\lan \Real \A\xi,\xi\ran + \lan \Real \A\eta,\eta\ran
       +2 \lan (p^{-1}\Imm\A + p'^{-1}\Imm\A^{*}) \xi,\eta \ran +\\
       \lan \Imm({\bf b} + {\bf c}),\eta\ran 
       -
       2 \lan \Real (\alpha{\bf b}/p - \beta {\bf c}/p'),\xi\ran
       +\\
       \Real\left[\dive\left((1-\alpha){\bf b}/p -  (1-\beta){\bf c}/p'\right) - a
       \right]\geq 0
\end{gathered}
\end{equation}
       for any $\xi,\eta\in\RR^{n}$,
    the operator $A$ is $L^{p}$-dissipative.
\end{corollary}

Generally speaking,  conditions \eqref{polyn} are not
    necessary, as the following example shows. 
   \begin{example}     
        let $n=2$ and
	$$
	\A=\left(\begin{array}{cc}
	1 & i\gamma \\ -i\gamma & 1
	\end{array}\right)
	$$
	where $\gamma$ is a real constant, ${\bf b}={\bf c}=a=0$. In this case
polynomial \eqref{polyn}
	is given by
	$$
	(\eta_{1}-\gamma\xi_{2})^{2}+	(\eta_{2}-\gamma\xi_{1})^{2}
-(\gamma^{2}-4/(pp'))|\xi|^{2}.
	$$
	
	Taking $\gamma^{2} >
	4/(pp')$,  condition \eqref{polyn} is not satisfied,
	while we have the $L^{p}$-dissipativity, because the corresponding
operator $A$
	is the Laplacian.
    \end{example}
    
    Note that in this 
    example the matrix
    $\Imm\A$ is not symmetric. Later we 
    give another example showing that, even for symmetric matrices
    $\Imm\A$,
    conditions \eqref{polyn} are not necessary 
    for $L^{p}$-dissipa\-tivity (see Example \ref{ex:1}).
Nevertheless in the next section
    we show that the conditions are necessary for the $L^{p}$-dissipativity, 
    provided the operator
    $A$ has no lower order terms and the matrix
    $\Imm\A$ is symmetric (see Theorem \ref{th:1new} and Remark \ref{rm:1}).

In the case of an operator \eqref{eq:N} without lower order
terms:
\begin{equation}
    Au=\dive(\A\nabla u)
    \label{eq:nolower}
\end{equation}
with the coefficients $a^{hk}\in
(\Cspt(\Om))^{*}$,
we can give
an algebraic necessary and sufficient condition for the
$L^{p}$-dissipativity.

\begin{theorem}[\cite{cialmaz1}]\label{th:1new}
    Let the matrix $\Imm\A$ be symmetric, i.e. $\Imm\A^{t}=\Imm\A$. The  form
$$
        \elle(u,v)=\int_{\Om}\lan \A\nabla u,\nabla v\ran
$$
    is $L^{p}$-dissipative if and only if
    \begin{equation}\label{eq:24}
	|p-2|\, |\lan \Imm\A\xi,\xi\ran| \leq 2 \sqrt{p-1}\,
	\lan \Real\A\xi,\xi\ran
    \end{equation}
	for any $\xi\in\RR^{n}$, where $|\cdot|$ denotes the total
	variation.
\end{theorem}

\begin{remark} \label{rm:1} One can prove that condition \eqref{eq:24} holds if and only if
\begin{equation}\label{eq:condnew}
\frac{4}{p\,p'}\lan \Real\A\xi,\xi\ran +\lan \Real\A\eta,\eta\ran
-2(1-2/p)\lan\Imm\A\xi,\eta\ran \geq 0
\end{equation}
for any $\xi,\eta\in\RR^{n}$. This means that   conditions
\eqref{polyn} are
necessary and sufficient for the operators
considered in Theorem \ref{th:1new}.
\end{remark}

\begin{remark}
    Let us assume that either $A$ has lower order terms
	or they are absent and $\Imm\A$ is not symmetric.
Using the same arguments as in Theorem \ref{th:1new},
one could prove that \eqref{eq:24} (or, equivalenty, \eqref{eq:condnew}) is
still a necessary condition
for $A$ to be $L^{p}$-dissipative. However, in general, it is not
sufficient. This is shown by the next example (see also
Theorem \ref{th:const} below  for the particular case of constant
coefficients).
\end{remark}

\begin{example} Let $n=2$ and let $\Om$ be a bounded domain. Denote by $\sigma$ 
a  not identically vanishing real function in $\Cspt^{2}(\Om)$ and let 
$\lambda\in\RR$.
Consider operator \eqref{eq:nolower} with
$$
\A=\left(\begin{array}{cc}
	1 & i\lambda\de_{1}(\sigma^{2}) \\ -i\lambda\de_{1}(\sigma^{2}) & 1
	\end{array}\right)
$$
i.e.
$$
Au=\de_{1}(\de_{1}u+i\lambda\de_{1}(\sigma^{2})\, \de_{2}u) + 
\de_{2}(-i\lambda\de_{1}(\sigma^{2})\,
\de_{1}u+\de_{2}u),
$$
where $\de_{i}=\de/\de x_{i}$ ($i=1,2$).

By definition, we have $L^{2}$-dissipativity if and only if
$$
\Real\int_{\Om}((\de_{1}u+i\lambda\de_{1}(\sigma^{2})\, \de_{2}u)
\de_{1}\overline{u}
+ (-i\lambda\de_{1}(\sigma^{2})\,
\de_{1}u+\de_{2}u)\de_{2}\overline{u})\, dx \geq 0
$$
for any $u\in\Cspt^{1}(\Om)$, i.e. if and only if
$$
\int_{\Om}|\nabla u|^{2}dx -2\lambda\int_{\Om}\de_{1}(\sigma^{2})
\Imm(\de_{1}\overline{u}\,\de_{2}u)\, dx \geq 0
$$
for any $u\in\Cspt^{1}(\Om)$.
Taking $u=\sigma\, \exp(itx_{2})$ ($t\in\RR$), we obtain, in particular,
\begin{equation}
    t^{2}\int_{\Om}\sigma^{2}dx -t\lambda 
    \int_{\Om}(\de_{1}(\sigma^{2}))^{2} dx  + \int_{\Om}|\nabla \sigma|^{2}dx \geq 0.
    \label{eq:inpartic}
\end{equation}

Since
$$
\int_{\Om}(\de_{1}(\sigma^{2}))^{2} dx > 0,
$$
we can choose $\lambda\in\RR$ so that
\eqref{eq:inpartic} is impossible
for all $t\in\RR$.
Thus  $A$ is not $L^{2}$-dissipative, although
\eqref{eq:24} is satisfied.

Since $A$ can be written as
$$
	Au=\Delta u - i\lambda(\de_{21}(\sigma^{2})\,
	\de_{1}u - \de_{11}(\sigma^{2})\, \de_{2}u),
$$
the same example shows
that \eqref{eq:24} is not sufficient for the $L^{2}$-dissipativity
in the presence of lower order terms, even if $\Imm\A$ is symmetric.
\end{example}

Generally speaking, it is impossible
to obtain an algebraic
characterization for
an
operator with
lower order terms.
Indeed, let us consider,
for example, the operator
$$
Au=\Delta u + a(x)u
$$
in a bounded domain ${\Omega}\subset{\mathbb{R}}^{n}$
with zero Dirichlet
boundary data. Denote by ${\lambda}_{1}$
the first eigenvalue of the Dirichlet problem for
the Laplace equation
in ${\Omega}$. A sufficient condition for
the $L^{2}$-dissipativity of $A$
has the form
$\operatorname{Re} a\leq {\lambda}_{1}$, and we cannot
give an algebraic characterization of
${\lambda}_{1}$.

Consider, as another example, the operator
$$
A =\Delta + \mu
$$
where $\mu$ is a nonnegative Radon measure on ${\Omega}$.
The operator $A$ is $L^{p}$-dissipative if and only if
\begin{equation}
\int_{{\Omega}}|w|^{2}d\mu \leq \frac{4}{p\,p'}\int_{{\Omega}}|\nabla w|^{2}dx
\label{eq:schro}
\end{equation}
for any $w\in {C_{0}}^{\infty}({\Omega})$ (cf. Lemma
\ref{lemma:1}).
Maz'ya \cite{mazya62, mazya64, mazyasimon}
proved
that the following condition is sufficient for
\eqref{eq:schro}:
\begin{equation}
\frac{\mu(F)}{\hbox{\rm cap}_{{\Omega}}(F)} \leq \frac{1}{p\,p'}
\label{eq:schro1}
\end{equation}
for all compact set $F\subset {\Omega}$
and the following condition is necessary:
\begin{equation}
\frac{\mu(F)}{\hbox{\rm cap}_{{\Omega}}(F)} \leq \frac{4}{p\,p'}
\label{eq:schro2}
\end{equation}
for all compact set $F\subset {\Omega}$.
Here, $\hbox{\rm cap}_{{\Omega}}(F)$ is the capacity of $F$
with respect  to
${\Omega}$, i.e.,
\begin{equation}\label{eq:defcap}
\hbox{\rm cap}_{{\Omega}}(F)=\inf\Big\{\int_{{\Omega}}|\nabla u|^{2}dx\, :\
u\in {C_{0}}^{\infty}({\Omega}),\ u\geq 1\ \hbox{\rm on}\ F\Big\}.
\end{equation}

The condition \eqref{eq:schro1}
is not necessary and the condition \eqref{eq:schro2} is not sufficient.

However, it is possible to find necessary and
sufficient conditions in
the case of constant coefficients.
Namely, let $A$ be the differential operator
\begin{equation*}
    Au=\nabla^{t}(\A\nabla u) + {\bf b}\nabla u + au
\end{equation*}
with constant complex coefficients. Without loss of generality
 we 
assume that the matrix $\A$ is symmetric.

\begin{theorem}[\cite{cialmaz1}]\label{th:const}
    Let $\Om$ be an open set in $\RR^{n}$ which contains balls of
    arbitrarily large radius. The operator $A$
    is $L^{p}$-dissipative if and only if there exists a real constant
    vector $V$ such that
    \begin{gather*}
        2\Real \A V+\Imm {\bf b}=0
 \\
	\Real a +  \lan \Real\A V,V\ran \leq 0 
    \end{gather*}
    and the inequality
    \begin{equation}\label{eq:V1}
	|p-2|\, |\lan \Imm\A\xi,\xi\ran| \leq 2 \sqrt{p-1}\,
	\lan \Real\A\xi,\xi\ran
    \end{equation}
   holds for any $\xi\in\RR^{n}$.
\end{theorem}

\begin{corollary}[\cite{cialmaz1}]\label{cor:const}
    Let $\Om$ be an open set in $\RR^{n}$ which contains balls of
       arbitrarily large radius. Let us suppose that the matrix $\Real\A$ 
       is
       not singular. The operator $A$
       is $L^{p}$-dissipative if and only if \eqref{eq:V1} holds and
       \begin{equation}
           4\Real a \leq -\lan (\Real\A)^{-1}\Imm {\bf b}, \Imm {\bf b}\ran\, .
           \label{eq:newV2}
       \end{equation}
\end{corollary}

\begin{example}\label{ex:1}
Let $n=1$ and $\Om=\RR^{1}$. Consider the operator
$$
\left( 1 + 2\, \frac{\sqrt{p-1}}{p-2}\, i\right) u'' +2i u' -u,
$$
where $p\neq 2$ is fixed. Conditions
\eqref{eq:V1} and \eqref{eq:newV2} are satisfied
 and this operator is $L^{p}$-dissipative, in view of Corollary \ref{cor:const}.

On the other hand, the polynomial considered in Corollary \ref{cor:1} is
$$
Q(\xi,\eta)=\left(2\, \frac{\sqrt{p-1}}{p}\, \xi -\eta\right)^{2}
+2\eta+1
$$
which is not nonnegative for any $\xi,\eta\in\RR$.
This shows that, in general, condition \eqref{polyn} is not necessary for the
$L^{p}$-dissipativity, even if the matrix $\Imm\A$ is symmetric.
\end{example}

\section{$L^p$-dissipativity for systems}\label{sec:systems}

In this Section we describe  criteria we have obtained for
some systems of partial differential equations.

\subsection{Systems of the first order}

Let $\B^h$  and 
$\Cm^{h}$ ($h=1,\ldots,n$) be
$m\times m$ matrices whith complex-valued
entries $b^h_{ij}, c^{h}_{ij} \in (C_{0}(\Om))^{*}$
($1\leq i,j \leq m$).
Let $\Dm$ stand for a matrix whose elements
$d_{ij}$ are complex-valued distributions
in $(C^{1}_{0}(\Om))^{*}$.

Let $\elle(u,v)$ be the sesquilinear form
\begin{equation}\label{formsyst1agen}
\elle(u,v) = 
\int_{\Om} \lan \B^{h}\de_{h}u,v \ran - \lan \Cm^{h}u,\de_{h}v\ran
+ \lan \Dm u,v\ran
\end{equation}
defined in $(C^{1}_{0}(\Om))^{m}\times (C^{1}_{0}(\Om))^{m}$,
where $\de_{h}=\de/\de x_{h}$.

The form $\elle$ is related to the 
system of partial differential operators of the first order:
\begin{equation}\label{syst1agen}
    E u = \B^{h}\de_{h}u + \de_{h}(\Cm^{h}u) + \Dm u
\end{equation}

As in the scalar case,
we say that the form $\elle$ is $L^{p}$-dissipative if
\eqref{eq:defdis1}-\eqref{eq:defdis2} hold 
for all $u\in (C^1_0(\Om))^{m}$.


We have found necessary and sufficient conditions for
the $L^p$-dissipativity when
\begin{equation*}
    E u = \B^{h}\de_{h}u  + \Dm u
\end{equation*}
and the entries of the matrices $\B^{h}, \Dm$ are locally integrable
functions. Moreover we suppose that also $\de_h\B^{h}$ (where the derivatives are in the sense of distributions) is a matrix with locally integrable entries.

\begin{theorem}[\cite{cialmaz17}]\label{th:main}
    The form
    $$
    \elle(u,v) = 
\int_{\Om} \lan \B^{h}\de_{h}u,v \ran 
+ \lan \Dm u,v\ran
    $$ is $L^{p}-dissipative$
    if, and only if, the following conditions are satisfied:
    \begin{enumerate}
\item     \begin{gather}
    \B^{h}(x) = b_{h}(x)\, I, \qquad \text{if } p\neq 2,
        \label{eq:n=1pn2x} \\
        \B^{h}(x) = (\B^h)^*(x), \qquad \text{if } p=2,
        \label{eq:n=1p=2x}
    \end{gather}
    for almost any $x\in\Om$ and $h=1,\ldots,n$. Here 
    $b_{h}$ are real 
      locally integrable functions ($1\leq h \leq n$).
\item    \begin{equation}
        \Real \lan (p^{-1} \de_{h}\B^{h}(x) - \Dm(x)) \zeta,\zeta\ran  \geq 0
        \label{eq:disug}
    \end{equation}
    for any $\zeta\in \CC^{m}$, $|\zeta|=1$ and for almost any $x\in\Om$.
    \end{enumerate}
\end{theorem}

As far as the more general operator \eqref{syst1agen} is concerned,
we have the following result, under the assumption that $\B^{h}$, 
$\Cm^{h}$,
$\Dm$, $\de_{h}\B^{h}$ and $\de_{h}\Cm^{h}$ are matrices with complex locally
integrable entries.

\begin{theorem}[\cite{cialmaz17}]
The form \eqref{formsyst1agen} is $L^p$-dissipative if, and only if,
the following conditions are satisfied
    \begin{enumerate}
\item \begin{gather*}
    \B^{h}(x)+ \Cm^h(x) = b_{h}(x)\, I, \qquad \text{if } p\neq 2,
   \\
        \B^{h}(x) + \Cm^h(x) = (\B^h)^*(x) + (\Cm^h)^*(x), \qquad \text{if } p=2,
    \end{gather*}
    for almost any $x\in\Om$ and $h=1,\ldots,n$. Here 
    $b_{h}$ are real 
      locally integrable functions ($1\leq h \leq n$).
\item    \begin{equation*}
        \Real \lan (p^{-1} \de_{h}\B^{h}(x) - p'^{-1}\de_{h}\Cm^{h}(x) - \Dm(x)) \zeta,\zeta\ran  
        \geq 0
            \end{equation*}
    for any $\zeta\in \CC^{m}$, $|\zeta|=1$ and for almost any $x\in\Om$.
    \end{enumerate}
\end{theorem}

\subsection{Systems of the second order}\label{subsec:system2}

In this section we consider the class of systems of partial differential
equations of the form
\begin{equation}
Eu = \de_h(\A^{h}(x)\de_h u)  + 
\B^{h}(x)\de_{h}u  + \Dm(x)u,
\label{eq:Ecomp}
\end{equation}
where $\A^{h}$, $\B^h$ and $\Dm$ are $m\times m$ matrices with 
complex locally integrable
entries.

If the operator \eqref{eq:Ecomp} has no lower order terms, we have
necessary and sufficient conditions:

\begin{theorem}[\cite{cialmaz2}]
The operator
$$
\de_h(\A^{h}(x)\de_h u) 
$$
is $L^p$-dissipative if and only if
\begin{equation}\label{eq:cond2ord}
\begin{gathered}
 \Real \lan \A^{h}(x) \lambda,\lambda\ran-(1-2/p)^{2}\Real\lan 
  \A^{h}(x)\om,\om\ran (\Real \lan\lambda,\om\ran)^{2}
   \cr
    -
    (1-2/p)\Real(\lan \A^{h}(x)\om,\lambda\ran -
    \lan \A^{h}(x)\lambda,\om\ran)
       \Real \lan \lambda,\om\ran   \geq 0
\end{gathered}
\end{equation}
for  almost every $x\in\Om$ and for
   every $\lambda,\om\in\CC^{m}$, $|\om|=1$, $h=1,\ldots,n$.
\end{theorem}

Combining this result with Theorem \ref{th:main} we find

\begin{theorem}[\cite{cialmaz17}]
Let $E$ be the operator \eqref{eq:Ecomp},
where $\A^{h}$ are $m\times m$ matrices with 
complex locally integrable
entries and the matrices $\B^h(x)$, $\Dm(x)$ satisfy the hypothesis of Theorem \ref{th:main}.
If \eqref{eq:cond2ord} holds 
for  almost every $x\in\Om$ and for
   every $\lambda,\om\in\CC^{m}$, $|\om|=1$, $h=1,\ldots,n$, and if
conditions \eqref{eq:n=1pn2x}-\eqref{eq:n=1p=2x} and \eqref{eq:disug} are satisfied,
the operator $E$ is $L^p$-dissipative.
\end{theorem}

Consider now the operator \eqref{eq:Ecomp} in the scalar
case (i.e. $m=1$)
$$
\de_h(a^h(x) \de_h u) + b^h(x) \de_h u + d(x) u
$$
($a^h, b^h$ and $d$ being scalar functions).
In this case such an operator can be written in the form
\begin{equation}
Eu=\dive (\A(x) \nabla u) + \B(x)\nabla u + d(x)\, u
\label{eq:scalar}
\end{equation}
where  $\A=\{c_{hk}\}$, $c_{hh}=a^h$, $c_{hk}=0$ if $h\neq k$ and
$\B=\{b^h\}$.
For such an operator one can show that 
 \eqref{eq:cond2ord} is equivalent to
 \begin{equation}
 \frac{4}{pp'}\lan \Real \A(x) \xi,\xi\ran + \lan \Real \A(x) \eta,\eta\ran
- 2(1-2/p)\lan \Imm \A(x) \xi,\eta\ran \geq 0
\label{eq:23}
\end{equation}
for almost any $x\in\Om$ and for any $\xi,\eta\in\RR^n$ (see \cite[Remark 4.21, p.115]{cialmazbook}).
Condition \eqref{eq:23} is in turn equivalent to the inequality:
 \begin{equation}
|p-2|\, |\lan \Imm\A(x)\xi,\xi\ran| \leq 2 \sqrt{p-1}\,
	    \lan \Real\A(x)\xi,\xi\ran
\label{eq:24x}
\end{equation}
for almost any $x\in\Om$ and for any $\xi\in\RR^n$  (see \cite[Remark 2.8, p.42]{cialmazbook}).
We have then
\begin{theorem}[\cite{cialmaz17}]
Let $E$ be the scalar operator \eqref{eq:scalar} where $\A$ is a diagonal matrix. 
If inequality \eqref{eq:24x} and 
conditions \eqref{eq:n=1pn2x}-\eqref{eq:n=1p=2x} and \eqref{eq:disug} are satisfied,
the operator $E$ is $L^p$-dissipative.
\end{theorem}

More generally, consider the scalar operator \eqref{eq:scalar} with
a matrix $\A=\{a_{hk}\}$ not necessarily diagonal. 
The following result holds true.

\begin{theorem}[\cite{cialmaz17}]
Let the matrix $\Imm \A$ be symmetric.
If inequality \eqref{eq:24x} and 
conditions \eqref{eq:n=1pn2x}-\eqref{eq:n=1p=2x} and \eqref{eq:disug} are satisfied,
the operator \eqref{eq:scalar} is $L^p$-dissipative.
\end{theorem}

\section{The $L^p$-dissipativity of the Lam\'e operator}
\label{sec:elast}

Let us consider the classical operator of 
linear elasticity
\begin{equation}
    Eu=\Delta u + (1-2\nu)^{-1}\nabla \dive u
    \label{opelast}
\end{equation}
where  $\nu$ is the Poisson ratio. 
We assume that
either $\nu>1$ or $\nu<1/2$. 
It is well known that $E$ 
 is strongly elliptic if and only if this condition is
 satisfied.
 
 We remark that the elasticity system is not of the form considered
 in the subsection \ref{subsec:system2}.
 
 Let $\elle$ be the bilinear form associated with
operator \eqref{opelast}, i.e.
$$
  \elle(u,v)=- \int_\Om (\lan\nabla u , \nabla v\ran + (1-2\nu)^{-1}\dive u\, 
    \dive v)\, dx\, ,
$$

The following lemma holds in any dimensions:

\begin{lemma}[\cite{cialmaz2}]\label{lemma:5}
    Let $\Om$ be a domain of $\RR^{n}$.
    The operator \eqref{opelast} is $L^{p}$-dissipative if and only if
    \begin{equation*}
    \int_{\Om}[C_{p}|\nabla|v||^{2}-\sum_{j=1}^{2}|\nabla v_{j}|^{2}
    +\gamma\, C_{p}\, |v|^{-2}|v_{h}\de_{h}|v||^{2}-\gamma\, |\dive v|^{2}]\, 
    dx \leq 0
\end{equation*}
for any $v\in (\Cspt^{1}(\Om))^{2}$, where
\begin{equation}
    C_{p}=(1-2/p)^{2}, \qquad \gamma= (1-2\nu)^{-1}.
    \label{cpg}
\end{equation}
\end{lemma}

More precise results are known in the case of 
planar elasticity.
At first we have an algebraic necessary condition:

\begin{lemma}[\cite{cialmaz2}]\label{th:5}
    Let $\Om$ be a domain of $\RR^{2}$.
    If the operator \eqref{opelast}  is $L^{p}$-dissipative, we have 
    \begin{equation*}
	    C_{p}[|\xi|^{2}+\gamma\, 
	    \lan \xi,\om\ran^{2}]\lan\lambda,\om\ran^{2}
	    - |\xi|^{2}|\lambda|^{2} - \gamma\, \lan\xi,\lambda\ran^{2}\leq 0
	\end{equation*} 
	for any $\xi,\, \lambda,\, \om\in\RR^{2}$, $|\om|=1$
	(the constants $C_{p}$ and $\gamma$ being given by \eqref{cpg}).
\end{lemma}

Hinging on Lemmas \ref{lemma:5} and \ref{th:5}, we proved
\begin{theorem}[\cite{cialmaz2}]
Let $\Om$ be a domain of $\RR^{2}$.
   The operator \eqref{opelast} is $L^{p}$-dissipative if
   and only if 
  \begin{equation}
      \left(\frac{1}{2}-\frac{1}{p}\right)^{2} \leq 
      \frac{2(\nu-1)(2\nu-1)}{
	    (3-4\nu)^{2}}\, . 
       \label{condvecchia}
   \end{equation}
\end{theorem}

Concerning the elasticity system in any dimension,
the next Theorem shows that condition
\eqref{condvecchia} is necessary, even in the case
of a non constant Poisson ratio.
Here $\Om$ is a bounded domain in $\RR^n$ whose boundary is in the
class $C^2$.

\begin{theorem}[\cite{cialmaz3}]
    Suppose $\nu=\nu(x)$ is a continuos function defined in
    ${\Om}$ such that 
    $$
    \inf_{x\in\Om}|2\nu(x)-1|>0.
    $$ 
    If the operator \eqref{opelast} is $L^{p}$-dissipative in $\Om$, then
     \begin{equation*}
      \left(\frac{1}{2}-\frac{1}{p}\right)^{2} \leq 
      \inf_{x\in\Om} \frac{2(\nu(x)-1)(2\nu(x)-1)}{(3-4\nu(x))^{2}}\, . 
   \end{equation*}
\end{theorem}

We do not know if condition \eqref{condvecchia} 
is sufficient for the $L^p$-dissipativity of the $n$-dimensional elasticity.
The next Theorem provides a more strict sufficient condition.
 
\begin{theorem}[\cite{cialmaz3}]
   Let $\Om$ be a domain in $\RR^n$.  If
    \begin{equation*}
        (1-2/p)^{2}\leq
\begin{cases}
\displaystyle\frac{1-2\nu}{2(1-\nu)} & \text{if}\ \nu<1/2\\
\\
\displaystyle\frac{2(1-\nu)}{1-2\nu}& \text{if}\ \nu > 1.
\end{cases}
    \end{equation*}
    the operator \eqref{opelast} is $L^p$-dissipative.
\end{theorem}

We have also a kind of weighted $L^p$-negativity of elasticity system defined on rotationally symmetric vector functions.

Let $\Phi$ be a point on the $(n-2)-$dimensional 
unit sphere $S^{n-2}$ with spherical coordinates 
$\{\vartheta_{j}\}_{j=1,\ldots,n-3}$ and $\varphi$, where $\vartheta_{j}\in 
(0, \pi)$ and $\varphi \in [0, 2\pi)$. A point $x \in\RR^{n}$ is 
represented as a triple $(\varrho,\vartheta, \Phi)$, where $\varrho > 0$ and
$\vartheta\in [0,\pi]$. 
Correspondingly, a vector $u$ can be written as $u =
(u_{\varrho},u_{\vartheta},u_{\Phi})$ with 
$u_{\Phi} = (u_{\vartheta_{n-3}},\ldots,u_{\vartheta_{1}},u_{\varphi})$.
We call $u_{\varrho},u_{\vartheta},u_{\Phi}$ the spherical components
of the vector $u$.

\begin{theorem}[\cite{cialmaz3}]
    Let the spherical components 
    $u_{\vartheta}$ and $u_{\Phi}$ of the vector $u$  vanish, i.e. 
    $u=(u_{\varrho},0,0)$, and let $u_{\varrho}$ depend only
    on the variable $\varrho$. Then, if $\alpha \geq n-2$, we have
    \begin{equation*}
       \int_{\RR^{n}}\left(\Delta u + (1-2\nu)^{-1}  \nabla \dive 
	u\right) |u|^{p-2}u\, \frac{dx}{|x|^{\alpha}}\leq 0
    \end{equation*}
    for any  $u\in 
    (\Cspt^{\infty}(\RR^{n}\setminus \{0\}))^n$ satisfying the aforesaid
symmetric conditions,
     if and only if
    \begin{equation*}
        -(p-1)(n+p'-2) \leq \alpha \leq n+p-2.
    \end{equation*}
    If $\alpha<n-2$ the same result holds replacing 
    $(\Cspt^{\infty}(\RR^{n}\setminus \{0\}))^n$
    by $(\Cspt^{\infty}(\RR^{n}))^n$.
    \end{theorem}

    \section{The $L^p$-dissipativity of the ``complex oblique
    derivative'' operator}\label{sec:oblique}

    In this section we  consider the $L^{p}$-dissipativity
    of the ``complex oblique
    derivative'' operator, i.e. of the boundary
    operator
\begin{equation}\label{eq:odp}
\lambda\cdot\nabla u =
\frac{\partial u}{\partial x_{n}} +
\sum_{j=1}^{n-1}a_{j}\frac{\partial u}{\partial x_{j}}\, ,
\end{equation}
the coefficients $a_{j}$ being  $L^{\infty}$ complex valued functions
defined on $\RR^{n-1}$.

We start with a Lemma, in which we use the concept of multiplier
(see \cite{mazmult}). In particular we  consider the space - we denote by
${\mathscr M}$ - of the multipliers acting form  $H^{1/2}(\RR^{n-1})$
into itself. Necessary and sufficient conditions  for a function to be a multiplier and equivalent expressions for the relevant norms are given in \cite{mazmult} (see, in particular, Theorem 
4.1.1, p.134).

\begin{lemma}\label{lem:mult}
Let $a=(a_1,\ldots,a_{n-1})$ be a vector multiplier 
belonging to ${\mathscr M}$. We have
\begin{equation*}
\left| \int_{\RR^{n-1}} a_j\, f\, \partial_j g\, dx' \right| \leq
\Vert a\Vert_{\mathscr M} \Vert \nabla f\Vert_{L^2(\RR^n_+)}
\Vert \nabla g\Vert_{L^2(\RR^n_+)}
\end{equation*}
for any $f, g \in  H^1(\RR^n_+)$, $j=1,\ldots,n-1$. Here the derivatives
are understood in the sense of distributions.
\end{lemma}

\begin{proof}
Let us denote by $\Lambda$ the operator $\sqrt{-\Delta}$ and write
$$
\int_{\RR^{n-1}} a_j\, f\, \partial_j g\, dx' = 
\int_{\RR^{n-1}} \Lambda^{1/2}(a_j\, f)\,  \Lambda^{-1/2}(\partial_j g)\, dx'\, .
$$

We have
\begin{gather*}
\left| \int_{\RR^{n-1}} a_j\, f\, \partial_j g\, dx' \right| \leq
\Vert \Lambda^{1/2}(a_j\, f) \Vert_{L^2(\RR^{n-1})}
\Vert \Lambda^{-1/2}(\partial_j g) \Vert_{L^2(\RR^{n-1})} \leq \\
\Vert a\Vert_{\mathscr M} 
\Vert \Lambda^{1/2}f \Vert_{L^2(\RR^{n-1})}
\Vert \Lambda^{1/2}g  \Vert_{L^2(\RR^{n-1})}
\leq
\Vert a\Vert_{\mathscr M} \Vert \nabla f\Vert_{L^2(\RR^n_+)}
\Vert \nabla g\Vert_{L^2(\RR^n_+)}\, .
\end{gather*}
\end{proof}

The next Theorem provides a sufficient condition for the $L^p$-dissipativity of operator \eqref{eq:odp} under the assumption that
\begin{equation}\label{eq:condM}
\Vert \Imm a \Vert_{\mathscr M} <
\frac{4}{p\, p'}\, .
\end{equation}

\begin{theorem}\label{th:6.2}
Suppose condition \eqref{eq:condM} is satisfied.
If there exists a real vector $\Gamma\in L^2_{\text{loc}}(\RR^n)$ such that
\begin{equation}
    -\partial_{j}(\Real a_{j}) \, \delta(x_n) \leq \frac{p}{2}\left(\frac{4}{p\, 
    p'}- \Vert \Imm a \Vert_{\mathscr M}\right) ( \dive \Gamma - 
    |\Gamma|^{2})
    \label{eq:verbmaz}
\end{equation}
in $\RR^n$, in the sense of distributions, then the operator \eqref{eq:odp} is $L^p$-dissipative.
\end{theorem}

\begin{proof}
It is well known that  $-\partial /\partial_{x_n}u(x',0)=\Lambda(u)$
where, as in the previous Lemma, $\Lambda=\sqrt{-\Delta}$.
This permits us to introduce the sesquilinear form
\begin{equation}\label{eq:newdefelle}
\elle(u,v)=-\int_{\RR^{n-1}}\lan \Lambda^{1/2} u,
 \Lambda^{1/2}v \ran\, dx' +
 \int_{\RR^{n-1}} \lan\sum_{j=1}^{n-1}a_j \partial_j u, v\ran\, dx'\, .
\end{equation}

We say that the operator $\lambda\cdot\nabla$ is 
$L^{p}$-dissipative if conditions 
\eqref{eq:defdis1} and \eqref{eq:defdis2} are satisfied
for any $u\in \Cspt^{1}(\RR^{n-1})$,
the form $\elle$ being given by \eqref{eq:newdefelle}.

Suppose $p\geq 2$. Denote by $U$ the harmonic extension of $u$,
i.e.
$$
U(x',x_n)= \frac{2}{\omega_{n}}\int_{\RR^{n-1}}u(y')\,
\frac{x_{n}}{(|x'-y'|^{2}+x_{n}^{2})^{n/2}}\, dy'
$$
$\omega_n$ being the measure of the unit sphere in $\RR^n$.

Integrating by parts we get
\begin{gather*}
\Real\elle(u,|u|^{p-2}u) = \\
    \Real \int_{\RR^{n-1}} \frac{\partial U}{\partial x_n}\,
|u|^{p-2} \overline{u}\, dx' + 
\Real \int_{\RR^{n-1}} \sum_{j=1}^{n-1} a_{j}\frac{\partial u}{\partial x_j}
\, |u|^{p-2} \overline{u}\, dx' =\\
-\Real \int_{\RR^{n}_+}\nabla U \cdot \nabla (|U|^{p-2} \overline{U})\, dx + 
\Real \int_{\RR^{n-1}} \sum_{j=1}^{n-1} a_{j}\frac{\partial u}{\partial x_j}
\, |u|^{p-2} \overline{u}\, dx'\, .
\end{gather*}

Therefore, for $p\geq 2$, the operator $\lambda\cdot\nabla$ is $L^{p}$-dissipative if
and only if
\begin{equation}\label{eq:ghjk}
\Real \int_{\RR^{n-1}} \sum_{j=1}^{n-1} a_{j}\frac{\partial u}{\partial x_j}
\, |u|^{p-2} \overline{u}\, dx' \leq
\Real \int_{\RR^{n}_+}\nabla U \cdot \nabla (|U|^{p-2} \overline{U})\, dx
    \end{equation}
for any $u\in  \Cspt^{1}(\RR^{n-1})$, $U$ being the harmonic extension
of $u$ to $\RR^n_+$. 

Setting $V=|U|^{(p-2)/2}U$ and $v=|u|^{(p-2)/2}u$, we get
$$
 \, |u|^{p-2}\overline{u}\, \partial_j u=
-(1-2/p)\,|v|\,\partial_j |v| + \overline{v}\, \de_j v
$$
and then
$$
\Real(|u|^{p-2}\overline{u}\, \partial_j u)=
-(1-2/p)\,|v|\,\partial_j |v| + \Real (\overline{v}\, \de_j v
) = \frac{1}{p} \partial_j (|v|^2).
$$

With similar computations we find
$$
\Real(\nabla U \cdot \nabla (|U|^{p-2} \overline{U})) =
|\nabla V|^2 - (1-2/p)^2 |\, \nabla|V|\, |^2
$$

Inequality \eqref{eq:ghjk} becomes
\begin{equation}\label{eq:ghjk2}
\begin{gathered}
-\frac{1}{p} \int_{\RR^{n-1}} \partial_j(\Real a_j)\,  |v|^2\, dx'
-\int_{\RR^{n-1}}\Imm a_j\, \Imm(\overline{v}\, \partial_j v)\, dx' \leq
\\
\int_{\RR^{n}_+}(|\nabla V|^2 - (1-2/p)^2 |\, \nabla|V|\, |^2)\, dx\, .
\end{gathered}
\end{equation} 

Lemma \ref{lem:mult} implies that
\begin{equation}\label{eq:ineqlemma}
\left| \int_{\RR^{n-1}}\Imm a_j\, \Imm(\overline{v}\, \partial_j v)\, dx' \right| \leq \Vert \Imm a \Vert_{\mathscr M} \int_{\RR^n_+}
|\nabla V|^2 dx
\end{equation}

On the other hand, inequality \eqref{eq:verbmaz} is the necessary and sufficient condition for the validity of the inequality
\begin{equation}\label{eq:speri-1}
-\int_{\RR^{n-1}} \partial_j(\Real a_j) \,  |v|^2\, dx'
\leq 
\frac{p}{2} \left(\frac{4}{p\, 
    p'}- \Vert \Imm a \Vert_{\mathscr M}\right)\int_{\RR^n}
    |\nabla V|^2\, dx
\end{equation}
for any $V\in\Cspt^\infty(\RR^n)$, $v$ being the restriction of $V$
on $\RR^{n-1}$ (see \cite[Th. 5.1]{mazverb}).
In particular, we find
\begin{equation}\label{eq:speri}
-\int_{\RR^{n-1}} \partial_j(\Real a_j) \,  |v|^2\, dx'
\leq 
p\left(\frac{4}{p\, 
    p'}- \Vert \Imm a \Vert_{\mathscr M}\right)\int_{\RR^n_+}
    |\nabla V|^2\, dx
\end{equation}
for any function $V\in\Cspt^\infty(\RR^n)$ which is even with
respect to $x_n$.
Since $|\, \nabla|V|\, |\leq |\nabla V|$ and $1 - (1-2/p)^2 = 
4/(p\, p')$, we have also
$$
\frac{4}{p\, p'} \int_{\RR^n_+} |\nabla V|^2\, dx \leq
\int_{\RR^{n}_+}(|\nabla V|^2 - (1-2/p)^2 |\, \nabla|V|\, |^2)\, dx\, . 
$$

This inequality, together with \eqref{eq:ineqlemma} and \eqref{eq:speri},
show that \eqref{eq:ghjk2} holds and 
the operator $\lambda\cdot\nabla$ is $L^{p}$-dissipative.

If $1<p<2$ 
we have to show that  
$$
\Real\elle(|u|^{p'-2}u,u) \leq 0
$$
for any $u\in  \Cspt^{1}(\RR^{n-1})$.

Arguing as for \eqref{eq:ghjk} we find that the operator $\lambda\cdot\nabla$ is $L^{p}$-dissipative if
and only if
$$
\Real \int_{\RR^{n-1}} \sum_{j=1}^{n-1} a_{j}\frac{\partial (|u|^{p'-2}u)}{\partial x_j}
\, \overline{u}\, dx' \leq
\Real \int_{\RR^{n}_+}\nabla (|U|^{p'-2}U) \cdot \nabla \overline{U}\, dx\, .
$$

Setting $V=|V|^{p'-2}V$ and $v=|u|^{p'-2}v$, we have
\begin{gather*}
\Real(|u|^{p'-2}u\, \partial_j \overline{u})=
(1-2/p')\,|v|\,\partial_j |v| + \Real (v\, \de_j \overline{v}
) = \frac{1}{p} \partial_j (|v|^2),\\
\Real(\nabla (|U|^{p'-2}U) \cdot \nabla \overline{U}) = 
|\nabla V|^2 - (1-2/p)^2 |\, \nabla|V|\, |^2.
\end{gather*}

Therefore the operator $\lambda\cdot\nabla$ is $L^p$-dissipative if
and only if \eqref{eq:ghjk2} holds and
the proof proceeds as in the case $p\geq 2$.
\end{proof}

    The next Theorem provides a necessary condition,
    similar to the previous one, 
    but which contains a different constant. We remark
    that in the next Theorem we do not require the
    smallness of $\Vert \Imm a \Vert_{\mathscr M}$.
    
    \begin{theorem}
    If the operator $\lambda\cdot\nabla$ is $L^{p}$-dissipative, then there exists 
a real vector $\Gamma\in L^2_{\text{loc}}(\RR^n)$ such that we have the inequality
$$
    -\partial_{j}(\Real a_{j}) \, \delta(x_n) \leq \frac{p}{2}\left(1 +\Vert \Imm a \Vert_{\mathscr M}\right) ( \dive \Gamma - 
    |\Gamma|^{2})
$$
    in the sense of distributions.
    \end{theorem}
    
  \begin{proof}
  If $\lambda\cdot\nabla$ is $L^{p}$-dissipative, we have the inequality \eqref{eq:ghjk}. Keeping in mind
  \eqref{eq:ineqlemma} we find
\begin{gather*}
-\frac{1}{p} \int_{\RR^{n-1}} \partial_j(\Real a_j)\,  |v|^2\, dx' \leq 
\int_{\RR^{n}_+}(|\nabla V|^2 - (1-2/p)^2 |\, \nabla|V|\, |^2)\, dx + \\
\int_{\RR^{n-1}}\Imm a_j\, \Imm(\overline{v}\, \partial_j v)\, dx' \leq 
 (1+\Vert \Imm a\Vert_{\mathscr M})\int_{\RR^{n}_+}|\nabla V|^2\, dx\, . 
\end{gather*}

Let us now consider $V\in\Cspt^\infty(\RR^n)$ and write
$V=V_o +V_e$, where $V_o$ and $V_e$ are odd and even respectively.
The last inequality we have written leads to
\begin{gather*}
-\frac{1}{p} \int_{\RR^{n-1}} \partial_j(\Real a_j)\,  |v|^2\, dx' \leq
 \frac{1+\Vert \Imm a\Vert_{\mathscr M}}{2}\int_{\RR^{n}}|\nabla V_o|^2\, dx \leq \\
  \frac{1+\Vert \Imm a\Vert_{\mathscr M}}{2}\int_{\RR^{n}}|\nabla V|^2\, dx 
\end{gather*}
and the thesis follows from the already quoted result \cite[Th. 5.1]{mazverb}.

  \end{proof}

 A different sufficient condition can be obtained
 by using the concept of capacity (see \eqref{eq:defcap}).
 
\begin{theorem}
 Suppose condition \eqref{eq:condM} is satisfied. If
  \begin{equation}\label{eq:condcap}
-\frac{1}{\hbox{\rm cap}_{{\Omega}}(F)}\int_{F\cap\RR^{n-1}} \partial_j(\Real a_j)\, dx'
\leq \frac{p}{8}\left(\frac{4}{p\, 
    p'}- \Vert \Imm a \Vert_{\mathscr M}\right) 
\end{equation}
for all compact sets $F\subset \RR^n$, then
the operator $\lambda\cdot\nabla$ is $L^{p}$-dissipative.  
\end{theorem}
\begin{proof}
We know that inequality \eqref{eq:schro} holds for any test function $w$ if condition \eqref{eq:schro1} is satisfied for all compact sets $F\subset \Omega$. As remarked in \cite[Remark 5.2]{mazyasimon}, \eqref{eq:schro1} implies \eqref{eq:schro} even without the requirement that $\mu\geq 0$, i.e. $\mu$ can be an arbitrary locally finite real valued charge.

Therefore, condition \eqref{eq:condcap} implies inequality
\eqref{eq:speri-1} and, as in Theorem  \ref{th:6.2}, the result follows.
\end{proof}

 In the case of the real oblique derivative problem we have a necessary and sufficient condition.

 \begin{theorem}
    Let us suppose that the coefficients $a_j$ in \eqref{eq:odp} are real valued ($j=1,\ldots,n-1$).  The operator $\lambda\cdot\nabla$ is $L^{p}$-dissipative if and only if
there exists a real vector $\Gamma\in L^2_{\text{loc}}(\RR^n)$ such that
\begin{equation}\label{verbmazreal}
    -\partial_{j}(\Real a_{j}) \, \delta(x_n) \leq \frac{2}{p'}( \dive \Gamma - 
    |\Gamma|^{2})
\end{equation}
in the sense of distributions.
\end{theorem}
\begin{proof}
The operator $\lambda\cdot\nabla$ being real, 
we consider the conditions \eqref{eq:defdis1} and \eqref{eq:defdis2}  for
real valued functions.
We remark that, if $V$ is a real valued function, we
have $|\nabla V|=|\nabla |V||$ and then 
$$
\int_{\RR^{n}_+}(|\nabla V|^2 - (1-2/p)^2 |\, \nabla|V|\, |^2)\, dx
= \frac{4}{p\, p'} \int_{\RR^n_+} |\nabla V|^2 dx
\, . 
$$

Therefore the $L^p$-dissipativity of  $\lambda\cdot\nabla$
occurs if and only if 
\begin{equation}\label{eq:ghjkreal}
\begin{gathered}
-\frac{1}{p} \int_{\RR^{n-1}} \partial_j(\Real a_j)\,  |v|^2\, dx'
 \leq
\frac{4}{p\, p'} \int_{\RR^n_+} |\nabla V|^2 dx\, .
\end{gathered}
\end{equation} 
Arguing as in the proof of Theorem \ref{th:6.2}, we find that
\eqref{eq:ghjkreal} holds if and only if 
$$
-\int_{\RR^{n-1}} \partial_j(\Real a_j)\,  |v|^2\, dx'
 \leq
\frac{2}{p'} \int_{\RR^n} |\nabla V|^2 dx\, .
$$
for any $V\in \Cspt^1(\RR^{n})$. 
Appealing again  to  \cite[Th. 5.1]{mazverb}, we see that
this inequality holds if and only if condition \eqref{verbmazreal}
is satisfied.
\end{proof}

\section{$L^p$-positivity of certain integral operators}\label{sec:intop}

The aim of this section is to prove the $L^p$-positivity of
the operator
\begin{equation}
    Tu(E)=\int_E\int^*_{\RR^n}[u(x)-u(y)]\, K(dx,dy)\,  .
    \label{defT}
\end{equation}
Here $E$ is a Borel set in $\RR^n$, $K(dx,dy)$ is a nonnegative Borel measure on $\RR^n\times \RR^n$, locally finite
outside the diagonal $\{(x,y)\in \RR^n\times \RR^n\ |\ x=y\}$
and the function $u$ belongs to $ \Cspt^{1}(\RR^n)$.

The integral  in \eqref{defT} has to be understood as
a principal value in the sense of Cauchy
$$
\lim_{\varepsilon\to 0^{+}}\int_E
\int_{\RR^n\setminus B_{\varepsilon}(x)}[u(x)-u(y)]\, K(dx,dy)\, 
$$
and we assume that such singular integral does exist for
any $u\in  \Cspt^{1}(\RR^n)$.

In what follows we shall make also the following assumptions
on the kernel $K(dx,dy)$:
\begin{enumerate}[(i)]
\item $K(E,F)=K(F,E)$ for any Borel sets $E,F \subset \RR^n$; \label{en:1}

\item\label{en:2b} for any compact set $E\subset \RR^n$ we have
$$
\iint_{E\times E} |x-y|^2K(dx,dy) < \infty\,;
$$

\item\label{en:3b} for any $R>0$ we have
$$
\int_{|x|<R}\int_{|y-x|>2R} K(dx,dy) < \infty\, .
$$
\end{enumerate}

As an example, consider the measure $K(dx,dy)=|x-y|^{-n-2s}dxdy$ ($0<s<1$);
it satisfies all the previous conditions and in this case the operator $T$ coincides, up to 
a constant factor, to the fractional power of Laplacian
$(-\Delta)^s$ (see, e.g., \cite[p.230]{cialmazbook}).

As in \eqref{eq:defdis1}-\eqref{eq:defdis2}, we say
that $T$ is $L^p$-positive if
\begin{equation}\label{eq:Tuup}
\begin{gathered}
\int_{\RR^n}\lan Tu,|u|^{p-2}u\ran  \geq 0, \quad \text{if } p\geq 2, \\
\int_{\RR^n}\lan T(|u|^{p'-2}u),u\ran  \geq 0, \quad \text{if } 1< p < 2,
\end{gathered}
\end{equation}
for any $u\in  \Cspt^{1}(\RR^n)$.

\begin{theorem}\label{th:7.1}
    Let $K(dx,dy)$ be a kernel satisfying the previous conditions. Then
    the operator \eqref{defT} is $L^p$-positive. More precisely we have
    the inequalities
  \begin{equation}\label{eq:disconK}
  \begin{gathered}
\int_{\RR^{n}}\lan Tu,|u|^{p-2}u\ran \geq
\frac{4}{p\, p'} \iint_{\RR^n\times \RR^n}  (|u(y)|^{p/2}- |u(x)|^{p/2})^2 K(dx,dy)\, , \text{if } p\geq 2;\\
\int_{\RR^n}\lan T(|u|^{p'-2}u),u\ran \geq
\frac{4}{p\, p'} \iint_{\RR^n\times \RR^n}  (|u(y)|^{p'/2}- |u(x)|^{p'/2})^2 K(dx,dy)\, , \text{if } 1< p < 2.
\end{gathered}
\end{equation}
\end{theorem}
\begin{proof}
Let us observe that, in view of \eqref{en:1}, we may write
\begin{equation}
    \int_{\RR^{n}}\lan Tu,v\ran  =
    \frac{1}{2} \iint_{\RR^n\times\RR^n}
    [u(x)-u(y)][v(x)-v(y)]\, K(dx,dy)
    \label{eq:Tuv}
\end{equation}
for any $u, v \in  \Cspt^{1}(\RR^n)$. In fact, since $u$ and $v$ have
compact support and  
thanks to conditions \eqref{en:2b} and \eqref{en:3b}, the integral in the right hand side of \eqref{eq:Tuv} is absolutely convergent.
Now we may appeal to the dominated convergence Theorem to obtain
\eqref{eq:Tuv}.

Let $p\geq 2$ and consider
\begin{gather*}
\int_{\RR^{n}}\lan Tu,|u|^{p-2}u\ran =\\
    \frac{1}{2} \iint_{\RR^n\times\RR^n}
    [u(x)-u(y)][|u(x)|^{p-2}u(x)-|u(y)|^{p-2}u(y)]\, K(dx,dy)
\end{gather*}
for any $u\in  \Cspt^{1}(\RR^n)$. Note that in this case $|u|^{p-2}u\in  \Cspt^{1}(\RR^n)$.
 Since 
$$
 (x-y)(|x|^{p-2}x-|y|^{p-2}y) \geq \frac{4}{p\,p'}\, (|x|^{p/2}-|y|^{p/2})^2
$$
  for any $x,y\in \RR$ 
 (see \cite[p.231]{cialmazbook}), we have that
 \eqref{eq:Tuup} holds
for any $u\in  \Cspt^{1}(\RR^n)$.

If $1<p<2$, the second  condition in \eqref{eq:Tuup} can be written as
$$
\int_{\RR^{n}}\lan Tv,|v|^{p'-2}v\ran   \geq 0
$$
for any $v\in  \Cspt^{1}(\RR^n)$ and the result follows as 
in the case $p\geq 2$ already considered.
\end{proof}

As a Corollary, we have that
under an additional  condition,
we have a lower estimate involving a Besov semi-norm.
 
\begin{corollary}
Let the  kernel $K(dx,dy)$ 
satisfy the conditions of Theorem \ref{th:7.1}. Moreover
suppose that there exist  $C >0$ and $s\in (0,1)$ such that
\begin{equation} \label{eq:s}
 K(dx,dy) 
 \geq C \frac{dxdy}{|x-y|^{n+2s}} \quad \text{on } 
 \RR^n\times\RR^n\, .
\end{equation}
Then we have
\begin{gather*}
\int_{\RR^{n}}\lan Tu, u\ran\, |u|^{p-2}  \geq
\frac{2C}{p\, p'} \ \Vert  \, |u|^{p/2}\Vert^2_{{\mathcal L}^{s,2}(\RR^n)}\, ,\quad  \text{if } p\geq 2;\\
\int_{\RR^n}\lan T(|u|^{p'-2}u),u\ran \geq
\frac{2C}{p\, p'} \ \Vert  \, |u|^{p'/2}\Vert^2_{{\mathcal L}^{s,2}(\RR^n)}\, ,\quad  \text{if } 1< p < 2,
\end{gather*}
where
 $$
 \Vert  v \Vert_{{\mathcal L}^{s,2}(\RR^n)}
  =  \left(\iint_{\RR^n\times \RR^n} |v(y)-v(x)|^2 \frac{dxdy}{|y-x|^{n+2s}}\right)^{1/2}.
 $$
\end{corollary}

\begin{proof}
The result follows immediately from \eqref{eq:disconK} and \eqref{eq:s}.
 \end{proof}

 We conclude with a remark. In Sections \ref{sec:oblique} and
 \ref{sec:intop}, the space  $\Cspt^{1}(\RR^n)$ was the class of admissible
 functions.  Actually, in these cases, we could extend this class
 and consider more general functions like, for example,
 compactly supported Lipschitz functions or even 
 bounded functions in proper Sobolev spaces.

\bibliographystyle{amsplain}

\begin{thebibliography}{10}

\bibitem{carbonaro}  Carbonaro A.,  Dragi\v{c}evi\'c, O.: Convexity of power functions and bilinear embedding for
divergence-form operators with complex coefficients, arXiv:1611.00653.


	    \bibitem{cialmaz1}  {Cialdea} A., {Maz'ya} V.: Criterion for the $L^p$-dissipativity of
second order differential operators with complex coefficients.
\textit{J. Math. Pures Appl.}, \textbf{84}; 1067--1100 (2005).

    \bibitem{cialmaz2} {Cialdea} A., {Maz'ya} V.: Criteria for the $L^p$-dissipativity of
systems of second order differential equations. \textit{Ricerche Mat.}, \textbf{55};  233--265
(2006).

\bibitem{cialmaz3}  {Cialdea} A., {Maz'ya} V.: 
$L^p$-dissipativity of the Lam\'e operator, \textit{Mem. Differ. Equ. Math. Phys.}, 
\textbf{60}, 111--133 (2013).

    \bibitem{cialmazbook} {Cialdea, A., Maz'ya, V.}:
     \textit{Semi-bounded Differential Operators,
Contractive Semigroups and Beyond},
{Operator Theory: Advances and Applications}, \textbf{243},
Birkh\"auser, Berlin (2014).
	   
	   \bibitem{cialmaz17} {Cialdea, A., Maz'ya, V.}: The $L^p$-dissipativity of first order partial differential
operators, \textit{Complex Var. Elliptic Equ.}, DOI: 10.1080/17476933.2017.1321638 (2017).
	   
	   
	    
	    \bibitem{dindos} {Dindo\v{s}}, M., {Pipher}, J.:	 Regularity theory for solutions to second order elliptic operators with complex coefficients and the $L^p$ Dirichlet problem,  arXiv:1612.01568v3.


 
 
\bibitem{mazverb}  Jaye, B., Maz'ya, V., Verbitsky, I.,: Existence and regularity of positive solutions of elliptic equations of Schrödinger type. \textit{J. Anal. Math.} 118 , no. 2, 577--621 (2012).


  
\bibitem{mazya62}
Maz'ya, V.G.:
The negative spectrum of the higher-dimensional {S}chr\"odinger
  operator (Russian).
\textit{Dokl. Akad. Nauk SSSR}
\textbf{144},
721--722
(1962);
English transl.: \textit{Sov. Math. Dokl.}
\textbf{3},
808--810
 (1962).


\bibitem{mazya64}
Maz'ya, V.:
On the theory of the
higher-dimensional Schr\"odinger operator (Russian).
\textit{Izv. Akad. Nauk SSSR Ser. Mat.} \textbf{28}, 1145--1172
(1964).


\bibitem{mazyasimon}
Maz'ya, V.:
Analytic criteria in
the qualitative spectral analysis of the Schr\"odinger
operator. \textit{Proc. Sympos. Pure Math.} \textbf{76}, no. 1,
257--288 (2007).


\bibitem{mazmult}
Maz'ya, V.,  Shaposhnikova, T.:
\textit{Theory of Sobolev multipliers.
With applications to differential and integral operators.} Grundlehren Math. Wiss., 337. Springer--Verlag, Berlin (2009).



\end{thebibliography}

\end{document}